\documentclass[12pt]{amsart}
\usepackage{amssymb}
\newtheorem{theorem}{Theorem}[section]
\newtheorem{corollary}[theorem]{Corollary}

\newtheorem{proposition}[theorem]{Proposition}

\theoremstyle{definition}

\theoremstyle{remark}
\newtheorem{remark}[theorem]{Remark}

\numberwithin{equation}{section}

\def\Label#1{\label{#1}}

\newcounter{stepsctr}

\newcommand{\dist}{\operatorname{dist}}

\newcommand{\mbb}{\mathbb}
\newcommand{\mc}{\mathcal}
\newcommand{\ol}{\overline}

\newcommand{\ti}{\textit}
\newcommand{\tu}{\textup}

\renewcommand{\(}{\left(}
\renewcommand{\)}{\right)}

\def\bC{\mathbb C}
\def\bR{\mathbb R}
\def\bD{\mathbb D}
\def\bT{\mathbb T}
\def\EL{\mathcal E\!\mathcal L}

\def\EB{\mathcal E\!\mathcal B}

\begin{document}

\title{Two-dimensional shapes and lemniscates
}
\author{P. Ebenfelt}
\address{Department of Mathematics, University of California at San Diego, La Jolla, CA 92093-0112}
\email{pebenfel@math.ucsd.edu}
\author{D. Khavinson}
\address{Department of Mathematics \& Statistics, University of South Florida, Tampa, FL 33620-5700}
\email{dkhavins@cas.usf.edu}
\thanks{The first two authors were partially supported by the NSF grants DMS-0701121 and  DMS-0855597, respectively. The second author learned about Kirillov's theorem during Prof. Alexander Vassili'ev's visit to the University of South Florida that
 was supported by the NSF grants DMS-075305, and the USF Grant for Graduate Education and Research in Computer Vision and Pattern Recognition.
 The authors are also indebted to Prof. Boris Shapiro for pointing out several relevant references.}

\author{H. S. Shapiro}
\address{Department of Mathematics, Royal Institute of Technology, Stockholm, Sweden 100 44}
\email{shapiro@math.kth.se}

\date{}


\maketitle

\section{Introduction}

The newly emerging field of vision and pattern recognition focuses on the study of 2-dimensional ``shapes'', i.e., simple, smooth, closed curves in the
plane. A common approach to describing shapes consists of defining ``natural'' distances between them, embedding the shapes into a metric space and then studying
 the mathematical structure of the latter. Of course, the resulting metric space must faithfully represent the continuous variability of shapes and
 reflect in their classification a similarity between them, i.e., not make a distinction between the shapes obtained from one another by scaling or
 translation (cf.\ \cite{BCMS} and many references therein). Also, one may consult \cite{SSJD} for extensions to analysis of ``surfaces'', more
 specifically, surfaces of a human face.

Another idea which has apparently been pioneered by A. Kirillov
\cite{Ki1,Ki2} and developed by Mumford and Sharon \cite{SM} and
many others, consists of representing each shape by its
``fingerprint'', an orientation-preserving diffeomorphism of the unit circle. In this
context every shape defines a unique equivalence class of such
diffeomorphisms (up to a right
composition with a M\"{o}bius automorphism of the unit disk onto
itself.) More precisely, let $\Gamma$ be a smooth, simple closed
curve (a Jordan curve) in $\bC$,  $\Omega_-$ the region enclosed by $\Gamma$ (i.e.\ the bounded component of $\bC\setminus \Gamma$), and
$\Omega_+:=\hat{\mbb{C}}\setminus\overline{\Omega}_-$, where $\hat \bC$ denotes the Riemann sphere $\hat\bC:=\bC\cup\{\infty\}$. Let
$\Phi_-:\mbb{D}\to\Omega_-$, $\Phi_+:\mbb{D}_+\to\Omega_+$, where
$\mbb{D}=\{|z|<1\}$ is the unit disk, $\mbb{D}_+$ its complement in $\hat\bC$, and
$\Phi_\pm$ are conformal maps (whose existence is guaranteed by the Riemann mapping theorem). We accept the normalization
$\Phi_+(\infty)=\infty$ and $\Phi'_+(\infty)>0$, where the latter means that $\Phi_+$ has a Laurent series expansion in neighborhood of infinity,
$$
\Phi_+(z)=az+\sum_{k=0}^\infty a_k z^{-k},
$$
with $a>0$.
From now on we shall
assume $\Gamma$ to be a $C^\infty$ curve, and hence $\Phi_\pm$ extend in a $C^\infty$ fashion to the boundaries of their respective domains. Let
$\mbb{T}=\partial\mbb{D}$ denote the unit circle and consider
$k:=\Phi^{-1}_+\circ\Phi_-:\mbb{T}\to\mbb{T}$, an orientation preserving
$C^\infty$-diffeomorphism of the unit circle onto itself.   We can
think of the derivative $k'$ as a $2\pi$-periodic function
 on $\mbb{R}$. We have then $k(x+2\pi)=k(x)+2\pi$ and $k'>0$. Obviously, $k$ is uniquely determined
by $\Gamma$ up to a M\"{o}bius automorphism of $\mbb{D}$, i.e.\ up to right composition $k\circ\phi$ with
\begin{equation}\Label{mobius}
\phi(z)=\lambda\frac{z-a}{1-\bar a z},\quad |\lambda|=1,\quad a\in\bD.
\end{equation}
The equivalence class of the diffeomorphism $k$ under the action of the M\"obius group of automorphisms \eqref{mobius} is called the {\it fingerprint} of $\Gamma$. Moreover, if $\tilde \Gamma$ denotes the curve $A(\Gamma)$, where $A$ is the affine transformation
\begin{equation}\Label{aff}
A(z)=az+b,\quad a>0,\quad b\in\bC,
\end{equation}
then the fingerprint of $\tilde \Gamma$ equals that of $\Gamma$, as is easily verified. Thus, we have a map $\mc{F}$ from the set of all smooth
Jordan curves $\Gamma$ modulo scaling and translation as in \eqref{aff}, {\it shapes}, into the set
of all orientation preserving diffeomorphisms $k$ of the circle modulo  M\"{o}bius
automorphisms \eqref{mobius} of the unit disk, {\it fingerprints}, . The following theorem was first explicitly stated in
\cite{Ki1,Ki2}, although as is noted in \cite{Ki1,Ki2,SM} it follows
more or less directly from the results of Ahlfors and Bers \cite{AB}
on solutions of Beltrami equation (cf.\ \cite{Tak:Teo}).

\begin{theorem}\Label{t:Kir}
The map $\mc{F}$ is a bijection.
\end{theorem}

Using this theorem and well-developed software packages, e.g., \cite{Dr}, Mumford and Sharon show how in principle one may recover (i.e., approximate)
 $\Gamma$ from its fingerprint $k$ and vice-versa, by approximating $\Gamma$ by polygons and using the Schwarz--Christoffel formula.
 The experimental data presented in \cite{SM} looks stunningly convincing.

Note in passing that if we relax significantly the smoothness hypothesis on $\Gamma$, the map $\Gamma\to k$ from closed curves to circle
\ti{homeomorphisms} is neither onto nor one-to-one (cf.\ \cite{Bi}).

In this paper, we present a somewhat ``ideologically'' different
explanation of why Kirillov's theorem is true, motivated by a well-known
theorem of Hilbert (cf.\ \cite[Ch.\ 4]{Wa}) stating that any smooth
curve can be approximated (with respect to the Hausdorff distance in the plane \eqref{Hausmetric}) by polynomial lemniscates. A
considerable advantage in this approach lies in the observation
that the fingerprint of a polynomial
lemniscate of degree $n$ is particularly simple, an $n^{\tu{th}}$
root of a (finite) Blaschke product of degree $n$ (Theorem \ref{lemnifp}). We then show that every smooth, orientation-preserving diffeomorphism of the circle can be
approximated in the $C^1$-norm by these former simple ones (Theorem \ref{approx}). To complete the picture in this approach, we show that a diffeomorphism of the unit circle given by the $n^{\tu{th}}$ root of a Blaschke product of degree $n$ is the fingerprint of a unique polynomial lemniscate of degree $n$ (Theorem \ref{bpasfp}). Although we
have not been able yet to reconstruct effectively a lemniscate from
its fingerprint on the circle, we still hope that via
associated finite Blaschke products lemniscates could serve as
natural and convenient ``coordinates'' in the enormous space of
smooth shapes.

In Section \ref{s:lemfp},  we discuss fingerprints of
polynomial lemniscates and show that the diffeomorphisms induced
by roots of finite Blaschke products approximate all smooth diffeomorphisms
of the circle. In Section \ref{s:bpasfp}, we prove that the diffeomorphisms
induced by roots of finite Blaschke products are fingerprints of
lemniscates. We end with some questions and remarks assembled in \S4.

\section{Lemniscates and their fingerprints}\Label{s:lemfp}

A (polynomial) lemniscate of degree $n$ is a subset $\Gamma\subset \bC$ of the form $$\{z\in \bC\colon |P(z)|=1\},$$ where $P(z)$ is a polynomial in $z$ of degree $n$. We let $\Omega_-:=\{z\in\bC\colon |P(z)|<1\}$ and $\Omega_+:=\hat\bC\setminus\overline{\Omega_-}=\{z\in \bC\colon|P(z)|>1\}\cup\{\infty\}$. An immediate consequence of the maximum modulus theorem is that $\Omega_+$ can have no bounded components and, hence, is a connected open subset containing a neighborhood of $\infty$ in $\hat \bC$. We shall say that $\Gamma$ is a {\it proper} lemniscate of degree $n$ if $\Gamma$ is smooth ($P'(z)\neq 0$ on $\Gamma$) and $\Omega_-$ is connected. Note that the interior $\Omega_-$ of a proper lemniscate of degree $n$ (or, for a general smooth lemniscate, each component of $\Omega_-$) is also simply connected, since its complement is connected.

Recall that the critical points of a polynomial $P(z)$ of degree $n$ are the zeros of its derivative $P'(z)$. Let $z_1,\ldots,z_{n-1}$ be the critical points (repeated according to their multiplicity) of $P(z)$. The values $w_1,\ldots,w_{n-1}$, where $w_k=P(z_k)$ for $k=1,\ldots,n-1$, are called the critical values of $P(z)$. The multiplicity of a critical value $w$ is the number of times it appears in the list $w_1,\ldots, w_{n-1}$. For a smooth lemniscate $\Gamma$ of degree $n$, the property of being proper can be characterized by the critical values of its defining polynomial.

\begin{proposition}\Label{propcrit} Let $P(z)$ be a polynomial of degree $n$ and assume that the lemniscate $\Gamma=\{z\in\bC\colon |P(z)|=1\}$ is smooth. The following are equivalent:

\smallskip
\noindent
{\rm (i)} The open set $\Omega_-=\{z\in\bC\colon |P(z)|<1\}$ is connected (i.e.\ $
\Gamma$ is a proper lemniscate of degree $n$).

\smallskip
\noindent
{\rm (ii)} All the critical values $w_1,\ldots,w_{n-1}$ of $P(z)$ satisfy $|w_k|<1$.
\end{proposition}

\begin{proof}  First note that no critical value can satisfy $|w_k|=1$ since $\Gamma$ is smooth. Let $r$ denote the number of critical points, counted with multiplicities, in $\Omega_-$. Statement (ii) above is then equivalent to $r=n-1$. The equivalence of (i) and (ii) is a simple consequence of the classical Riemann-Hurwitz formula. Indeed, let $\Omega_1,\ldots, \Omega_k$ denote the components of $\Omega_-$, and $d_j$ and $r_j$ for $j=1,\ldots,k$ the number of zeros and critical points, respectively, of $P$ in $\Omega_j$. We then have $d_1+\ldots+d_k=n$ and $r_1+\ldots+r_k=r$. Now, let $f_j$ denote the restriction of $P$ to $\Omega_j$. Each $f_j$ is then a $d_j$-to-1 ramified covering (proper analytic map) $f_j\colon \Omega_j\to \mathbb D$ with total ramification number $r_j$. Since each $\Omega_j$, as well as $\mathbb D$, is simply connected and hence has Euler characteric $1$, the Riemann-Hurwitz formula in this setting (see e.g.\ \cite {St}) states that
$$
-1=-d_j+r_j,\quad j=1,\ldots k.
$$
By summing over $j$, we obtain
$-k=-n+r$ or $r=n-k$. Thus, we have $r=n-1$ if and only if the number of components of $\Omega_-$ is one.
\end{proof}

Consider a proper lemniscate $\Gamma=\{z\in \bC\colon |P(z)|=1\}$ of degree $n$. Clearly, there is no loss of generality in assuming that the degree $n$ coefficient of $P(z)$ is real and positive, i.e.\
\begin{equation}\Label{polydef}
P(z)=a_nz^n+a_{n-1}z^{n-1}+\ldots + a_0,\quad a_n>0,\quad \text{{\rm $a_k\in \bC$ for $k=0,\ldots,n-1$.}}
\end{equation}
Let $\xi_1,\ldots,\xi_n\in \Omega_-$ be the zeros of $P$ (repeated with multiplicity), and  $\Phi_-:\bD\to\Omega_-$ a Riemann mapping. Then, $P\circ\Phi_-$ is an $n$-to-$1$ ramified covering $\bD\to\bD$ and hence must be a finite Blaschke product of degree $n$, i.e.,
\begin{equation}
\Label{polybp}
B_1(z):=(P\circ\Phi_-)(z)=e^{i\theta}\prod_1^n\frac{z-a_j}{1-\ol{a_j}z},
\quad a_j=\Phi_-^{-1}\(\xi_j\),\quad \theta\in \bR.
\end{equation}
Indeed, $\(P\circ
\Phi_-\)/B_1$ is analytic in $\mbb{D}$, does not vanish there, is
continuous in the closed disk and has modulus one on
$\mbb{T}:=\partial\mbb{D}$, and thus is a unimodular constant.

Now let $\Phi_+^{-1}:\Omega_+\to\mbb{D}_+$ be the conformal mapping from the exterior $\Omega_+\subset\hat \bC$ of $\Gamma$ onto the exterior of the unit disk $\mbb{D}_+\subset\hat\bC$, normalized by $\Phi_+^{-1}(\infty)=\infty$, $\(\Phi_+^{-1}\)'(\infty)>0$. We claim that
\begin{equation}\Label{nroot}
 \Phi_+^{-1}(w)=\sqrt[n]{P(w)},
\end{equation}
where we choose a suitable branch of the $n$th root (i.e.\ $\sqrt[n]{1}=1$) to comply with our normalization of $\(\Phi_+^{-1}\)'(\infty)>0$. To see this, we note that $P\circ\Phi_+$ has a pole of order $n$ at $\infty$, no other poles in $\bD_+$, and maps $\bD_+$ to itself sending the boundary $\mathbb T=\partial \bD_+$ to itself; i.e. $P\circ\Phi_+$ is a ramified $n$-to-1 covering $\bD_+\to \bD_+$. It follows, as above, that $B_2:=P\circ\Phi_+$ is a Blaschke product of degree $n$ and since $B_2$ has all its poles at $\infty$ we conclude that $B_2(z)=cz^n$, where $c$ is a unimodular constant. Since $(\Phi'_+)(\infty)>0$ and the highest degree coefficients $a_n$ of $P$ is positive, we deduce that in fact $c=1$ and $B_2(z)=z^n$. The identity \eqref{nroot} follows readily. Since the fingerprint of $\Gamma$ is given by $k=\Phi^{-1}_-\circ \Phi_+$, we obtain the following theorem.

\begin{theorem}\Label{lemnifp} Let $P$ be a polynomial of degree $n$ such that $\Gamma=\{z\in \bC\colon |P(z)|=1\}$ is a proper lemniscate of degree $n$. Let $\Omega_-$ be the interior of $\Gamma$ and $\Phi_-\colon \bD\to \Omega_-$ a Riemann mapping. Then, the fingerprint $k\colon \mathbb T\to \mathbb T$ of $\Gamma$ is given by
\begin{equation}\Label{nrootbp}
k(z)=\sqrt[n]{B(z)}
\end{equation}
where $B$ is the Blaschke product
\begin{equation}\Label{PcompPhi}
B(z)=e^{i\theta}\prod_{k=1}^n\frac{z-a_k}{1-\bar a_k z},\quad a_k=\Phi_-^{-1}(\xi_k),\quad \theta\in \bR,
\end{equation}
and $\xi_1,\ldots,\xi_n$ denote the zeros of $P$ repeated according to multiplicity.
\end{theorem}

In view of Hilbert's theorem (cf.\ \cite[Ch.\ 4]{Wa}) that every
smooth curve can be approximated by lemniscates in the Hausdorff
metric, which measures the distance between two curves $C_1$, $C_2$
as (cf.\ \cite{BCMS})
\begin{equation}\Label{Hausmetric}
\dist\(C_1,C_2\)=\sup_{z\in C_2}\inf_{w\in C_1}|z-w|+\sup_{z\in C_1}\inf_{w\in C_2}|z-w|,
\end{equation}
our next goal is to address the following two questions:
\begin{enumerate}
\item[(I)]  Do the  diffeomorphisms $k$ given by \eqref{nrootbp} approximate in some reasonable metric  all orientation preserving diffeomorphisms of
the unit circle?
\end{enumerate}

If the answer to (I) is in the affirmative, then to complete our
(alternative) approximate visualization of Theorem 1.1 we have to
answer the following question
\begin{enumerate}
\item[(II)]  Does each diffeomorphism \eqref{nrootbp} represent the fingerprint of a lemniscate?
\end{enumerate}

The following theorem answers (I).

\begin{theorem}\Label{approx}
The (algebraic) diffeomorphisms \eqref{nrootbp} approximate all
orientation preserving diffeomorphisms $\Psi$ of the circle
$\mbb{T}$ in the $C^1$-norm.
\end{theorem}

\begin{remark}
As usual, the $C^1$-norm on $\mbb{T}$ means
$$\|f\|_{C^1}=\sup\limits_{\theta}\(\left|f\(e^{i\theta}\)\right|
+\left|f'\(e^{i\theta}\)\right|\).
$$
\end{remark}

\begin{proof}
First, we note that it suffices to verify the theorem for
real-analytic diffeomorphisms $\Psi$ since the latter are dense in
the $C^1$-norm in the set of all diffeomorphisms of $\mbb{T}$. Let
$\Psi(\theta)=\exp(i\psi(\theta)),$ where the real-valued function
$\psi(\theta)$ is strictly monotone increasing and $\psi'>0$ is
$2\pi$-periodic. Of course, $\psi(\theta + 2\pi)=\psi(\theta)+2\pi$. To approximate $\psi(\theta)$ on $[0,2\pi]$ in $C^1$-norm by
arguments of functions in \eqref{nrootbp} it suffices to approximate
uniformly on $\mbb{T}$ a (positive) function
$\psi'$ such that  $\int_0^{2\pi}\psi' d\theta=2\pi$
by functions $$\frac1n\,\frac d{d\theta}\arg B\(e^{i\theta}\),$$
where
$B$ is a Blaschke product of degree $n$.
Note that a straightforward calculation yields
\begin{equation}
\label{eq2.10} \frac1{2\pi}\,\frac d{d\theta}\(\frac1n\arg
B\(e^{i\theta}\)\)
=\frac1{2\pi}\left\{\frac1n\sum_{j=1}^nP\(e^{i\theta},a_j\)\right\},
\end{equation}
where for
\begin{equation}
\label{eq2.11} z=re^{i\phi},\quad P\(e^{i\theta},z\)
:=\frac{1-r^2}{1+r^2-2r\cos(\theta-\phi)}
\end{equation}
denotes the Poisson kernel evaluated at $z$.

Following \cite{GSS}, we argue now as follows. Since
$\psi'$ is positive and real-analytic on $\mbb{T}$, we can
approximate it by a positive trigonometric polynomial
$h(\theta)=\sum\limits_{-N}^Na_ke^{ik\theta}>0$ maintaining the
normalization $\int_{\mbb{T}} h d\theta =2\pi$. Observe the
following:
\begin{enumerate}
\item[(i)]  We can consider $h(\theta)$ to be boundary values of a bounded positive harmonic function
 $H(r,\theta):=\sum\limits_{-N}^Na_kr^{-|k|}e^{ik\theta}$ in $\mbb{C}\smallsetminus\mbb{D}$ (in other words, replacing
 $z^k$, $\ol{z}^k$ by $\frac1{\ol{z}^k}$ and $\frac1{z^k}$ in the expansion of $h$ in terms of $z$, $\ol{z}$).

\item[(ii)]  $H(r,\theta)$ extends as a positive harmonic function to a slightly larger domain $\mbb{D}^R_-:=\{|z|>R,R<1\}$ ---
this is obvious in view of the continuity of $H$ and the compactness of $\bT$.
\end{enumerate}

Representing $h(\theta)=H\(e^{i\theta}\)|_{r=1}$ in $\mbb{D}^R_-$ via the Poisson integral of its boundary values on $\{|z|=R<1\}$
and taking into account the change in orientation we easily obtain
\begin{equation}
\label{eq2.12}
\begin{gathered}
h\(e^{i\theta}\)=H\(e^{i\theta}\)|_{r=1}
=\frac1{2\pi}\int_0^{2\pi}\frac{1-R^2}
{1+R^2-2R\cos(\theta-\varphi)}\,H\(Re^{i\varphi}\)d\varphi \\
=\int_{|z|=R, R<1}P\(e^{i\theta},z\)d\mu(z),
\end{gathered}
\end{equation}
where $\mu>0$ is a probability measure supported on a compact subset
of $\mbb{D}$, i.e., on the circle $\{|z|=R<1\}$. Every such measure
$\mu$ is a weak$^\ast$ limit of discrete atomic probability measures
with $n$ atoms, $n\to\infty$, having equal charges $1/n$ at these
atoms. This last observation together with \eqref{eq2.10} finishes
the proof of the assertion that $\psi'$, with
$\int\limits_0^{2\pi}\psi' d\theta=2\pi$, is uniformly approximable on $\bT$
by functions $$\frac1n\,\frac d{d\theta}\arg B\(e^{i\theta}\),$$ where
$B$ is a Blaschke product of degree $n$. The remaining part of the
theorem is easily derived from it, so we shall omit it.
\end{proof}

\section{Roots of Blaschke products as fingerprints}\Label{s:bpasfp}

In this section, we shall prove a converse to Theorem \ref{lemnifp}, which answers question (II) above:

\begin{theorem}\Label{bpasfp}
Let $B$ be a Blaschke product of degree $n$,
\begin{equation}\Label{bp2} B(z)=e^{i\theta}\prod_{j=1}^n\frac{z-a_j}{1-\ol{a}_jz},
\quad\left|a_j\right|<1.
\end{equation}
There there is a proper lemniscate $\Gamma\subset\bC$ of degree $n$ such that its fingerprint $k\colon \mathbb T\to \mathbb T$ is given by
\begin{equation}\Label{e:bpasfp}
k(z)=\sqrt[n]{B(z)}.
\end{equation}
If $\tilde \Gamma\subset\bC$ is any other $C^1$-smooth Jordan curve with the same fingerprint, then there is an affine linear transformation $T(z):=az+b$, with $a>0$ and $b\in \bC$, such that $\tilde \Gamma=T(\Gamma)$.
\end{theorem}

\begin{proof} The uniqueness part of the theorem is of course a consequence of Theorem \ref{t:Kir}, but for the readers' convenience we shall reproduce the simple proof here. Suppose that two $C^1$-smooth Jordan curves $\Gamma,\tilde\Gamma\subset \bC$ have the same fingerprint, and let $\Phi_-,\tilde \Phi_-, \Phi_+, \tilde\Phi_+$ be the corresponding Riemann mappings $\bD_-\to \Omega_-$, $\bD_-\to \tilde \Omega_-$, $\bD_+\to  \Omega_+$, $\bD_+\to \tilde \Omega_+$, respectively (following the notation introduced above in an obvious way). Since both Jordan curves are assumed $C^1$-smooth, all Riemann mappings extend continuously and homeomorphically to the corresponding boundaries. (For this conclusion, weaker conditions than that of being $C^1$ suffice, but mere continuity does not; cf.\ \cite{Bi}).)
 The fact that the two fingerprints are equal means that $\Phi^{-1}_+\circ\Phi_-=\tilde \Phi^{-1}_+\circ\tilde \Phi_-$ on $\mathbb T=\partial \bD_-=\partial \bD_+$. This can be rewritten as $\tilde \Phi_+\circ\Phi^{-1}_+=\tilde\Phi_-\circ\Phi^{-1}_-$ on $\Gamma=\partial \Omega_-$. We conclude that the conformal mapping $\tilde \Phi_+\circ\Phi^{-1}_+\colon \Omega_+\to \tilde \Omega_+$ can be extended as a conformal mapping $T\colon \hat\bC\to \hat\bC$ by defining it as $\tilde\Phi_-\circ\Phi^{-1}_-$ in $\overline{\Omega}_-$. Since $T(\infty)=\infty$ and $T'(\infty)>0$, we conclude that $T(z)=az+b$ with $a>0$ and $b\in \bC$. This proves the uniqueness modulo affine linear transformations of the type described in the theorem.

To prove the existence part of the theorem, we shall consider a modification of the map $\mathcal F$ defined in the introduction in the setting of lemniscates and Blaschke products. A proper lemniscate $\Gamma$ of degree $n$ is the set of points $z\in \bC$ that satisfy $|P(z)|=1$, where $P$ is a polynomial in $z$ of degree $n$ whose highest order coefficient is positive (see \eqref{polydef}) and all of whose critical values belong to $\bD$ (see Proposition \eqref{propcrit}). It follows immediately from \eqref{nroot} that the polynomial $P$ is uniquely determined by $\Gamma$. We shall let $\mathcal L$ denote the subset of $\bR\times\bC\times\ldots\times\bC$ (with $n$ factors of $\bC$) that, under the map
\begin{equation}\Label{coeffmap}
(a_n,a_{n-1},\ldots, a_0)\mapsto P(z):=a_nz^n+\ldots+a_0,
\end{equation}
yields polynomials whose lemniscates are proper of degree $n$. Clearly, $\mathcal L$ is open. One can also easily prove that $\mathcal L$ is connected by using Proposition \ref{propcrit} in the following way. Let $P$ be a polynomial of degree $n$ corresponding to a point in $\mathcal L$ and denote by $\Gamma$ the corresponding proper lemniscate of degree $n$. By Proposition \ref{propcrit}, the critical values of $P$ all have modulus less than one. Consider the lemniscates $\Gamma_R$ defined by $|P(z)|=R$, or equivalently by $|P_R(z)|=1$ where $P_R(z):=P(z)/R$, for $R\geq 1$. Clearly, the critical values of $P_R$ all belong to the open disk of radius $1/R\leq1$ and hence the lemniscates $\Gamma_R$
are proper lemniscates of degree $n$. Now, pick $r>0$ such that
\begin{equation}\Label{conn1}
na_n|z|^{n-1}-\left ((n-1) |a_{n-1}|z|^{n-2}+(n-2)|a_{n-2}||z|^{n-3}+\ldots+|a_1|\right)>0, \quad \forall |z|\geq r,
\end{equation}
 and then pick $R>0$ such that
\begin{equation}\Label{conn2}
\frac{1}{R}\left(a_n|z|^n+|a_{n-1}||z^{n-1}|+\ldots+|a_0|\right)\leq \frac {1}{2},\quad \forall |z|\leq r.
\end{equation}
Finally, consider the lemniscates $\Gamma_R^t$ defined by the polynomials
$$
P_R^t(z):=\frac{1}{R}\left(a_nz^n+t(a_{n-1}z^{n-1}+\ldots +a_0)\right ), \quad 0\leq t\leq 1.
$$
It follows immediately from \eqref{conn1} and \eqref{conn2} that the critical values of $P_R^t$, for $0\leq t\leq 1$, have modulus less than 1 and, hence, the $\Gamma_R^t$ are all proper lemniscates of degree $n$. Note that $\Gamma_R^0$ is a circle. Thus, we conclude that any proper lemniscate $\Gamma$ of degree $n$ can be deformed through proper lemniscates of degree $n$ to a circle. It follows that the open subset $\mathcal L$ is connected.

As explained in the introduction, the fingerprint of a shape remains unchanged under translations and scaling. In other words, given a proper lemniscate $\Gamma=\{z\colon |P(z)|=1\}$ of degree $n$ with fingerprint $k\colon \bT\to\bT$, let $\Gamma_{{ab}}$ be the image of $\Gamma$ under the inverse of the affine linear transformation $T_{ab}(z):=az+b$, with $a>0$ and $b\in \bR$, and let $k_{ab}$ denote its fingerprint. Then, we have $k=k_{ab}$. We shall consider the space of equivalence classes $\{\Gamma_{ab}\}_{a>0,\ b\in\bC}$ of proper lemniscates $\Gamma$ of degree $n$ under this action of the group $\mathcal G:=\{T(z)=az+b\colon a>0,\ b\in \bC\}$. Note that if $\Gamma$ is defined by $|P(z)|=1$, then $\Gamma_{ab}$ is defined by $|P_{ab}(z)|=1$, where
\begin{equation}\Label{actG}
\begin{aligned}
P_{ab}(z):=P(az+b) &=a_n(az+b)^n+\ldots+a_0\\
&= a^n a_nz^n+(na^nb+a_{n-1})z^{n-1}+\ldots+P(b).
\end{aligned}
\end{equation}
Thus, in each equivalence class there is a unique polynomial of the form
\eqref{polydef} with $a_n=1/n$ and $a_{n-1}=0$. In other words, we can parametrize the space of equivalence classes of proper lemniscates of degree $n$ by a subset  $\EL\subset \bC^{n-1}$ under the identification
\begin{equation}\Label{ELmap}
(a_{0},\ldots,a_{n-2})\mapsto P(z)=\frac{1}{n}z^n+a_{n-2}z^{n-2}+\ldots+a_0.
\end{equation}
The subset $\EL\subset\bC^{n-1}$ is clearly open and connected by the same arguments as above. We also note that there is a finite, $n$-to-1, holomorphic (polynomial) mapping $\Lambda\colon\bC^{n-1}\to \bC^{n-1}$ induced by the action of $\mathcal G$ and defined as follows. For $(\tilde a_1,\ldots,\tilde a_{n-1})\in\bC^{n-1}$, consider the polynomial
\begin{equation}\Label{Lambdamap1}
\tilde P(z):=\frac{1}{n}z^n+\tilde a_{n-1}z^{n-1}+\ldots+ \tilde a_1z.
\end{equation}
The image $\Lambda(\tilde a_1,\ldots,\tilde a_{n-1})$ are the coefficients $(a_0,\ldots, a_{n-2})$ of the unique polynomial of the form
\begin{equation}\Label{Lambdamap2}
P(z)=\frac{1}{n}z^n+a_{n-2}z^{n-2}+\ldots+a_0
\end{equation}
in the equivalence class of $\tilde P$ under the action of $\mathcal G$ described in \eqref{actG}. (The reader may want to write down this map explicitly.) Let $\widetilde{\EL}$ denote the inverse image of $\EL$ under this map. For future reference, we remind the reader of the following well-known property, which will be used below, of a finite holomorphic mapping (or a branched covering) $H$ from one complex manifold $X$ to another $Y$ (endowed with some metrics). {\it The inverse images of $H$ in $X$ depend continuously on the values in $Y$} in the following sense: Let $w_0\in Y$ and let $H^{-1}(w_0)$ denote the (finite) set of inverse images. For any $\epsilon>0$, there exists $\delta>0$ such that if $w$ belongs to a $\delta$-ball centered at $w_0$, then the set of inverse images $H^{-1}(w)$ belongs to the union of $\epsilon$-balls centered at the points of $H^{-1}(w_0)$. We shall say that $H^{-1}(w)$ converges to $H^{-1}(w_0)$ as a set.

Next, consider the collection of Blaschke products of degree $n$,
\begin{equation}\Label{BP1}
B(z)=\lambda\prod_{k=1}^n\frac{z-b_k}{1-\bar b_k z},\quad b_k\in\bD,\quad \lambda\in \bT.
\end{equation}
Let $\mathcal M$ be the M\"obius group consisting of automorphisms of the unit disk
$$
\phi(z):=\lambda\frac{z-b}{1-\bar b z}, \quad |\lambda|=1,\quad b\in \bD,
$$
acting on Blaschke products by right composition. (Recall from the introduction that the fingerprint of a shape is only defined modulo this action on orientation preserving diffeomorphisms of $\bT$.) It is readily seen that each Blaschke product can be brought to one of the form
\begin{equation}\Label{eqbp}
B(z)=z\prod_{k=1}^{n-1}\frac{z-b_k}{1-\bar b_k z},\quad b_k\in\bD,
\end{equation}
by the action of $\mathcal M$. Also, each equivalence class of Blaschke products under this action contains a finite number ($n$ generically) of Blaschke products of this form. The Blaschke product in \eqref{eqbp} is of course invariant under permutations of the roots $(b_1,\ldots,b_{n-1})$, and hence the set of such Blaschke products  can be identified with the image $\mathcal B$ of $\bD^{n-1}:=\bD\times\ldots\times\bD\subset \bC^{n-1}$ under the finite holomorphic mapping $\bC^{n-1}\to\bC^{n-1}$
\begin{equation}\Label{rootmap}
(b_1,\ldots,b_{n-1})\mapsto (S_1(b),S_2(b),\ldots, S_{n-1}(b)),
\end{equation}
where $S_j(b)$ denotes the $j$th symmetric function on $n-1$ elements:
$$
\prod_{j=1}^{n-1}(z-b_j)=z^{n-1}+S_1(b)z^{n-2}+\ldots+S_{n-2}(b)z+S_{n-1}(b).
$$
Now, let $\EB$ denote the set of equivalence classes of Blaschke products under the action of $\mathcal M$, and let $\pi$ denote the projection of $\mathcal B$ onto $\EB$. Since the action of $\mathcal M$ on $\mathcal B\subset \bC^{n-1}$ is algebraic with only finitely many points in each equivalence class, $\EB$ is an algebraic variety (quotient singularity) of dimension $n-1$. Moreover, being the  image under successive continuous mappings (the finite mapping \eqref{rootmap} followed by $\pi$) of the connected space $\bD^{n-1}$, the space $\EB$ is connected.

By the discussion in Section \ref{s:lemfp}, we obtain a map $\mathcal F\colon \EL\to\EB$ as follows. For an element $e=(a_0,\ldots,a_{n-2})\in\EL$, let $P$ denote the corresponding polynomial $\eqref{ELmap}$, $\Gamma=\{z\in \bC\colon |P(z)|=1\}$ its proper lemniscate, $\Phi_{-}\colon \bD\to\Omega_-$ a Riemann map, and $B:=P\circ\Phi_-$ the corresponding Blaschke product (so that the fingerprint of $\Gamma$ is the $n$th root of $B$ by Theorem \ref{lemnifp}). We define $\mathcal F(e):=\pi(B)$. If we choose another Riemann map, we obtain another Blaschke product in the same equivalence class (and hence $\mathcal F$ is well defined). Also, any Blaschke product in this equivalence class can be produced by choosing a suitable Riemann map. Thus, to finish the existence part of Theorem \ref{bpasfp} it suffices to show that $\mathcal F$ is surjective. We first claim that $\mathcal F$ is continuous. To see this, let $\{e_k\}$ (notation as above) be a sequence of points in $\EL$ converging to $e_0$. Let $P_k, P_0$ be the corresponding polynomials and $\Gamma_k,\Gamma_0$ their lemniscates with interiors $\Omega_k,\Omega_0$. (Since we shall not need the exteriors in this argument, we shall omit the subscript "-" on the interiors and Riemann maps.) If we let $K$ be a closed disk that contains $\overline{\Omega}_0$ in its interior, then $P_k\to P_0$ uniformly on $K$ and, hence, $\Gamma_k\to \Gamma_0$ in the Hausdorff metric \eqref{Hausmetric}. Let us fix a $w\in \Omega_0$. Then, $w\in \Omega_k$ for $k$ sufficiently large. Now, let $\Phi_0\colon \bD\to \Omega_0$, $\Phi_k\colon \bD\to \Omega_k$ be Riemann mappings normalized by $\Phi_0(0)=\Phi_k(0)=w$ and $\Phi'_0(0)>0,\Phi'_k(0)>0$. By
a well-known theorem of Carath\'{e}odory (cf.\ \cite[Ch.\ II, Sec.\
5]{Go}) the Riemann mappings $\Phi_k:\mbb{D}\to\Omega_k$ converge
uniformly in $\mbb{D}$ to the Riemann mapping $\Phi_0:\mbb{D}\to\Omega_0$. Observe that the mapping taking the $n$ roots $(\xi_1,\ldots,\xi_n)$ of a monic polynomial to its $n$ coefficients $(a_0,\ldots,a_{n-1})$ is a finite holomorphic mapping (indeed, given by \eqref{rootmap} above modulo notation). Hence, the roots, as a set, depend continuously on the coefficients (in the sense explained above). It follows that the roots of $P_k$ converge to the roots of $P_0$ (again, as sets). We conclude that $B_k^{-1}(0)=(P_k\circ\Phi_k)^{-1}(0)$ converge to $B_0^{-1}(0)=(P_0\circ\Phi_0)^{-1}(0)$. This means that $B_k\to B_0$ in $\mathcal B$ and, hence, $\mathcal F(e_k)=\pi(B_k)\to \mathcal F(e_0)=\pi(B_0)$ in $\EB$. This proves that $\mathcal F$ is continuous.

Next, we observe that the map $\mathcal F$ is not injective. Indeed, the map as described above produces for each element $e\in \EL$ an equivalence class of a Blaschke product $B$ such that $B=k^n$, where $k$ is the finger print of the lemniscate $\Gamma$ associated to the point in $\EL$. We proved above that the map taking the point $e\in \EL$ to its fingerprint is injective. Thus, two lemniscates $\Gamma_1$ and $\Gamma_2$ corresponding to two points in $\EL$ will produce the same Blaschke product precisely when their fingerprints satisfy $k_2=\epsilon k_1$, where $\epsilon$ is a root of unity: $\epsilon^n=1$. Now, it is easy (and left to the reader) to verify that if $T(z)=\lambda z$, for some $|\lambda|=1$, then, for any shape $\Gamma$ with fingerprint $k$, the fingerprint of $T(\Gamma)$ is $\lambda k$. Let us introduce an equivalence relation on $\EL$ where two elements $e_1$ and $e_2$ are equivalent when their corresponding lemniscates $\Gamma_1$ and $\Gamma_2$ are related by $\Gamma_2=T(\Gamma_1)$ for some $T(z)=\epsilon z$ with $\epsilon^n=1$. If we let $\EL'$ denote the set of equivalence classes of elements in $\EL$, then by the comments above the map $\mathcal F$ factors as the map $\EL\to \EL'$ and an injective map $\mathcal F'\colon \EL'\to \EB$. As in the case of $\EB$ above, the set $\EL'$ is an algebraic variety of dimension $n-1$. The map $\mathcal F'\colon \EL'\to\EB$ is  continuous. To prove that $\mathcal F$ is surjective, we shall employ Koebe's continuity method based on Brouwer's ``invariance of a domain'' theorem (cf. \cite[Ch.\ 5,
Sec.\ 6]{Go}, see also \cite{Be}). Since $\mathcal F'$ is continuous and injective, Brouwer's theorem implies that
$\mc{F}'$ maps $\EL'$ homeomorphically onto an open subset of
$\EB$. Since the image $\mathcal F(\EL)$ clearly equals the image $\mathcal F'(\EL')$, we conclude that $\mathcal F(\EL)$ is an open subset of $\EB$. Since $\EB$ is connected as explained above, to prove that $\mathcal F$ is surjective it suffices to show that the image $\mathcal F(\EL)$ is closed in $\EB$.

We first note that a Blaschke product $B$ of degree $n$ is an $n$-to-1 branched covering of $\bD$ by itself. Thus, by the Riemann-Hurwitz formula as in the proof of Proposition \ref{propcrit}, $B$ has $n-1$ critical points (counted with multiplicity) in $\bD$. (This can also be easily seen by the argument principle, computing the change in argument of $B'(z)$ as $z$ traverses $\bT$ by noting that $B(z)$ circles $n$ times around $\bT$ as $z$ traverses $\bT$ once.) For a Blaschke product $B$ of the form \eqref{eqbp}, the critical values $c_1,\ldots,c_{n-1}$ are of course obtained by solving the equation $B'(z)=0$ and, hence, the critical values depend continuously (as sets) on the roots $b_1,\ldots, b_{n-1}$. Consequently, the critical values $w_1,\ldots, w_{n-1}\in \bD$ depend continuously on $b_1,\ldots,b_{n-1}$. Also, if $B=P\circ \Phi$ for some polynomial $P$ and conformal mapping $\Phi$, then clearly the critical values of $B$ and $P$ are the same. Moreover, the critical values of any two Blaschke products in the same equivalence class in $\EB$ are the same. It follows that the critical values of a polynomial corresponding to $e\in \EL$ and any representative of $\mathcal F(e)$ are the same.

Now, let $\{f_k\}$  be a sequence in $\mathcal F(\EL)$ converging to $f_0\in\EB$. We will show that $f_0=\mathcal F(e_0)$ for some $e_0\in \EL$, which will complete the proof. Let $e_k:=\mathcal F^{-1}(f_k)\in \EL$ and let $P_k$ denote the corresponding polynomials under the identification \eqref{ELmap}. In what follows, we shall abuse the notation and not distinguish between an element $e$ of coefficients and its corresponding polynomial $P$. Now, let $B_0$ be a Blaschke product in the equivalence class $f_0$ and let $w^{(0)}_1,\ldots,w^{(0)}_{n-1}$ denote the critical values (counted with multiplicity) of $B(z)$. Note that the critical values are independent of the choice of $B_0$. Let us choose $B_0\in\mathcal B$. Similarly, let
$w^{(k)}_1,\ldots,w^{(k)}_{n-1}$ denote the critical values of some (any) choice of Blaschke product $B_k\in \mathcal B$ in the equivalence class $f_k$. Since the map $\pi\colon \mathcal B\to \EB$ is an $n$-to-1 branched covering (each equivalence class in $\EB$ contains at most, and generically, $n$ distinct Blaschke products of the form \eqref{eqbp} as mentioned above), we may choose the $B_k\in \mathcal B$ such that $B_k\to B_0$ in $\mathcal B$. Since the critical values depend continuously on the roots of the Blaschke product, we can order the critical values $w^{(k)}_1,\ldots,w^{(k)}_{n-1}$ of $B_k$ in such a way that $w^{(k)}_j\to w^{(0)}_j$ as $k\to \infty$ for each $j$. (For instance, for each $k$ we can choose an ordering that minimizes the sum of the distances $|w^{(k)}_j-w^{(0)}_j|$.)

Now, recall the $n$-to-1 holomorphic mapping $\Lambda\colon \widetilde{\EL}\to \EL$, where $\widetilde \EL$ denote the polynomials of the form \eqref{Lambdamap1} whose critical values all belong to $\bD$. Note that each polynomial $\tilde P(z)$ of the form \eqref{Lambdamap1} is uniquely determined by the conditions that $\tilde P(0)=0$ and $\tilde P'(z)$ is monic. The map
\begin{equation}\Label{cptopoly}
(\zeta_1,\ldots,\zeta_{n-1})\mapsto \tilde P(z):=\int_{0}^z\left(\prod_{k=1}^{n-1}(z-\zeta_k)\right)dz
\end{equation}
is an $(n-1)!$-to-1 holomorphic map of $\bC^{n-1}$ onto the space of polynomials of the form \eqref{Lambdamap1}. By Theorem 1.2 in \cite{BCN}, the map sending the critical points of $\tilde P$ given by \eqref{cptopoly} to its critical values, i.e.\
\begin{equation}\Label{BCNmap}
\Psi(\zeta_1,\ldots,\zeta_{n-1}):= (\tilde P(\zeta_1),\ldots,\tilde P(\zeta_{n-1})),
\end{equation}
is a finite $n^{n-1}$-to-1 holomorphic mapping. Choose $\tilde P_k\in \widetilde{\EL}$ such that $\Lambda(\tilde P_k)=P_k$. Since $(\mathcal F\circ\Lambda)(\tilde P_k)=\mathcal F(P_k)=f_k$, it follows that the critical values of $\tilde P_k$ are $w^{(k)}_1,\ldots,w^{(k)}_{n-1}$. Thus, we can choose a sequence $\zeta_k=(\zeta^{(k)}_1,\ldots, \zeta^{(k)}_{n-1})\in \bC^{n-1}$ such that $\tilde P_k$ is given by the map in \eqref{cptopoly} and $\Psi(\zeta_k)=w_k:=(w^{(k)}_1,\ldots,w^{(k)}_{n-1})$. Now, choose $\zeta_0\in \bC^{n-1}$  such that $\Psi(\zeta_0)=w_0:=(w^{(0)}_1,\ldots,w^{(0)}_{n-1})$. Since $w_k\to w_0$ as $k\to \infty$, it follows that the the inverse images $\Psi^{-1}(w_k)$ converge to $\Psi^{-1}(w_0)$ as sets. Since $\zeta_k\in \Psi^{-1}(w_k)$, $\zeta_0\in \Psi^{-1}(w_0)$, and any $\Psi^{-1}(w)$ (in particular, $\Psi^{-1}(w_0)$) contains at most $n^{n-1}$ distinct preimages, it follows from the pigeon hole principle that there is a subsequence $\zeta_{k_j}$ that converges to $\zeta_0$. If we now let $\tilde P_0$ denote the image of $\zeta_0$ in $\widetilde \EL$ under the map \eqref{cptopoly}, then $\tilde P_{k_j}\to \tilde P_0$. Thus, if we denote by $P_0=\Lambda(\tilde P_0)\in\EL$, then $\mathcal F(P_k)=(\mathcal F\circ \Lambda)(\tilde P_{k_j})\to \mathcal F(P_0)$ as $j\to\infty$. Also, we have $\mathcal F(P_k)\to e_0$ and, hence, $\mathcal F(P_0)=e_0$ proving that $e_0$ belongs to the image of $\mathcal F$. We conclude that the image is closed. Since the image is also open and $\EB$ is connected, we conclude that $\mathcal F$ is surjective, which completes the proof.
\end{proof}

We would like to revisit the main idea in the proof above of the existence of a lemniscate with a prescribed $n$th root of a Blaschke product $B$ as its fingerprint. It is well known, and easily seen by the Riemann-Hurwitz formula or by noting that $B'(z)dz$ changes its argument by $2\pi n$ as $z$ traverses the unit circle, that $B$ has $n-1$ critical points $z_1,\ldots, z_{n-1}$ counted with multiplicities inside the unit disk. Let $w_j:=B(z_j)$, for $j=1,\ldots, n-1$, denote the critical values of $B$ in $\bD$. The main idea in the existence proof above is to look for a candidate of a lemniscate whose fingerprint could be $k=\sqrt[n]{B}$ among those given by $\Gamma:=\{z\in \bC\colon |P(z)|=1\}$, where the $P(z)$ are polynomials of degree $n$ whose critical values are $w_1,\ldots, w_{n-1}\in \bD$. We know that the number of equivalence classes of such lemniscates is finite and, by the uniqueness part in Theorem \ref{bpasfp} already proved, the map sending these equivalence classes to their fingerprints $k$ is injective. Thus, if we could show that the number of equivalence classes of Blaschke products with a given set of critical values $w_1,\ldots, w_{n-1}\in \bD$ is the same as (or at least does not exceed) the number of equivalence classes of polynomials with this set of critical values, then the map would be a bijection, which would complete the proof of the existence. We have, however, been unable to find a direct proof of this statement and unable to find it in the existing literature. To circumvent this obstacle, we instead use Koebe's continuity method as described in the proof above. As a biproduct of this alternative completetion of the proof, we are able {\it a posteriori} to get an accurate count of the number of equivalence class of Blaschke products with a given set of critical values in the following way.

We shall use the notation introduced in the proof of Theorem \ref{bpasfp}. 
If $w_1,\ldots, w_{n-1}$ are values in $\bD$, not necessarily distinct, then by Theorem 1.2 in \cite{BCN}  there are $n^{n-1}$ points $\zeta^1,\zeta^2,\ldots,\zeta^{n^{n-1}}\in\bC^{n-1}$, repeated according to multiplicity, such that $\Psi(\zeta^j)=w:=(w_1,\ldots,w_{n-1})$, where $\Psi$ is defined by \eqref{BCNmap}. Hence, if we disregard the ordering of the components of $\zeta^{j}=(\zeta^j_1,\ldots,\zeta^j_{n-1})$ and $w=(w_1,\ldots,w_{n-1})$, we conclude that there are $n^{n-1}$ polynomials $\tilde P$ of the form \eqref{Lambdamap1} in $\widetilde{\EL}$, again counted with multiplicity, whose critical values are $w_1,\ldots,w_{n-1}$. Since the map $\Lambda\colon \widetilde{\EL}\to\EL$ is $n$-to-1, there are $n^{n-2}$ polynomials in $\EL$ with critical values $w_1,\ldots,w_{n-1}$. Next, we observe that when $n=2$, the set of equivalence classes $\EL'$ introduced in the proof above coincides with $\EL$, but for $n\geq 3$ there are, generically, $n$ elements of $\EL$ in each equivalence class. Thus, the map $\mathcal F\colon \EL\to \EB$ is 1-to-1 when $n=2$, but $n$-to-1 for $n\geq 3$. We therefore obtain the following result, which seems to be of independent interest.

\begin{corollary} For $n\geq 3$ and any collection $w_1,\ldots,w_{n-1}$ of values in $\bD$, there are $n^{n-3}$ equivalence classes (counted with multiplicities) of Blaschke products in $\EB$ whose critical values (in $\bD$) are $w_1,\ldots,w_{n-1}$. For $n=2$, there is one equivalence class.
\end{corollary}

\section{Further remarks}

In this section, we collect some further remarks and observations regarding lemniscates and their fingerprints. We begin by proving the following rigidity results, which can be used to give an alternative proof of the uniqueness part (in the context of lemniscates) in Theorem \ref{bpasfp} but also seems to be of independent interest.

\begin{proposition}\Label{prop3.1}
Let $\Omega_-^{(1)}$, $\Omega_-^{(2)}$ be two domains
 bounded by proper lemniscates of degree $n$ and defined
by the equations $|P(z)|<1$, $|Q(z)|<1$, respectively, where $P$, $Q$ are of the
polynomials of degree $n$. Let $\left\{a_j^{(1)}\right\}_{j=1}^n$,
$\left\{a_j^{(2)}\right\}_{j=1}^n$ denote the respective nodes of
the lemniscates, i.e., the roots of $P$ and $Q$.
If $F:\Omega_-^{(2)}\to\Omega_-^{(1)}$ is a conformal map that maps the nodes of $\Omega_-^{(2)}$ onto the nodes of $\Omega_-^{(1)}$, then $F$ is an affine automorphism $F(w)=aw+b$ with $a,b\in \bC$.
\end{proposition}

\begin{remark}
Of course, we assume that if some nodes $a_j^{(2)}$ have non-trivial multiplicities, then $F$ preserves multiplicities as well.
\end{remark}

\begin{proof}[Proof of Proposition $\ref{prop3.1}$]
Consider $h(w):=Q(w)/P(F(w))$ defined in $\Omega_-^{(2)}$. By our hypothesis, $h$ is analytic and non-vanishing in $\Omega_-^{(2)}$ (all the zeros and poles are cancelled by the zeros of $Q$) and $|h|=1$ on the lemniscate $\Gamma_2:=\partial\Omega^{(2)}_-$. Moreover, since we can say the same about $1/h$, we conclude that $h$ is a unimodular constant. Hence, we have
\begin{equation}
\Label{eq3.71} P(F(w))=cQ(w),\quad |c|=1,
\end{equation}
in $\Omega_-^{(2)}$. Equation \eqref{eq3.71} implies that $F$ is an algebraic function and can therefore be continued analytically, maintaining equation \eqref{eq3.71}, along any curve in the Riemann sphere $\hat\bC$ avoiding a finite number of points $b_1,\ldots,b_m\in\hat\bC$. Since both $P$ and $Q$ have poles of order $n$ at $\infty$, the algebraic function $F$ can only take the value $\infty$ at $\infty$ and any continuation of $F$ will have a simple pole there. Thus, if we allow the continuation of $F$ to have poles (i.e.\ we continue $F$ as a meromorphic function or, equivalently, as an analytic map into $\hat\bC$), then the singularities $b_1,\ldots,b_m$ of $F$ are all branch points in $\bC$ at which a continuation of $F$ takes on finite values. Thus, if we continue $F$ along a curve $C$ from inside $\Omega^{(2)}_-$, ending at one of the branch points $b=b_j$, for some $j=1,\ldots, m$, (but avoiding the others) and $F$ develops a singularity at $b$, then the value $\zeta:=F(b)\in \bC$ must be a critical point of $P$. (Otherwise, $P$ would be locally biholomorphic near $\zeta$ and $F$ would not be singular there.) If we can show that there are no true branch points, i.e.\ the continuation of $F$ along $C$ to $b$ (for all branch points $b$ and all curves $C$ as described above) is analytic at $b$, then $F$ extends as a holomorphic map of $\hat \bC$ onto itself sending $\infty$ to itself with multiplicity one and no other poles, and is therefore of the form $F(w)=aw+b$, which would complete the proof of Proposition \ref{prop3.1}. Since $F$ is already analytic in a neighborhood of $\overline{\Omega^{(2)}_-}$, it suffices to show that there are no branch points in $\bC\setminus\overline{\Omega^{(2)}_-}$, so there is no loss of generality in assuming that $b\in\bC\setminus\overline{\Omega^{(2)}_-}$. Thus, assume, in order to reach a contradiction, that the continuation of $F$ along $C$ is singular at $b$ and that this continuation satisfies $\zeta=F(b)$. Then, as mentioned above, $\zeta$ is a critical point of $P$ and, hence by Proposition \ref{propcrit}, we have $\zeta\in \Omega^{(1)}_-$. This implies that $|P(F(b))|=|P(\zeta)|<1$, which implies that $|Q(b)|<1$. But then $b\in \Omega^{(2)}_-$, which is a contradiction since $b\not \in \Omega^{(2)}_-$ by assumption. This completes the proof of Proposition \ref{prop3.1}.
\end{proof}

We single out here the following observation, interesting
in its own right, that was used in the proof of
Proposition \ref{prop3.1} above.

\begin{proposition}
Let $P$ be a polynomial of degree $n$ and $F$ an algebraic function satisfying
\begin{equation}
\Label{eq4.1}
P\(F(w)\)=Q(w),
\end{equation}
where $Q$ is a polynomial. Let  $\Omega_-:=\{z\in \bC\colon|P(z)|<1\}$ and assume that $\Omega_-$ is connected (i.e.\  $\{z\in \bC\colon|P(z)|<1\}$ is a proper lemniscate of degree $n$).Then, all finite branch points of $F$ must lie inside $\{z\colon |Q(w)|<1\}$.
\end{proposition}

\begin{proof}  If $b\in \bC$ is a branch point of $F$ and there is a branch of $F$ that is singular at $b$, then the value $\zeta=F(b)$ of this branch is not $\infty$ (since $Q(b)\neq \infty$) and, as noted in the proof of Proposition \ref{prop3.1} above, $\zeta$ is a critical point of $P$. By Proposition \ref{propcrit}, we have $\zeta\in \Omega_-$ and, hence, $|Q(b)|=|P(F(b))|=|P(\zeta)|<1$, which completes the proof.
\end{proof}

We end this paper with two remarks.
(i) Since we have not been able, so far, to
find a constructive and simple way of identifying a lemniscate $\Gamma$
with a given fingerprint $k=\sqrt[n]{B}$, our approach does
seem to be inferior to that pioneered by Mumford and Sharon
\cite{SM}. However, taking into account the extremely simple form of
fingerprints of lemniscates ($\sqrt[n]{B}$, where $B$ is a
Blaschke product of degree $n$), perhaps, some way of efficiently combining
the two approaches may be useful. More precisely: 1) Approximate a given fingerprint
$k:\mbb{T}\to\mbb{T}$ by $\sqrt[n]{B}$ (as can be done arbitrarily well by Theorem \ref{approx}). 2) Find an approximate
shape of the lemniscate $\Gamma$ corresponding to $\sqrt[n]{B}$ by using the
technique from \cite{SM} based on Schwarz--Christoffel integrals. 3)
Approximate the shape obtained in \cite{SM} by a lemniscate as in
the proof of Hilbert's theorem which is quite constructive --- cf.\
\cite{Wa}. Hopefully, in the future, some numerical experiments
carried out along these lines will support our envisioning of the
(simple) fingerprints of lemniscates as natural ``coordinates''
in the space of shapes.

(ii) As a final remark in this paper, we would like to point out a direction of further study. If $R(z)$ is a rational function of degree $n$, then we may consider the rational lemniscate $\Gamma:=\{z\in \hat \bC\colon |R(z)|=1\}$. Under the assumption that the interior $\Omega_-:=\{z\in \hat\bC\colon |R(z)|<1\}$ is connected and simply connected (and there are no singularities on $\Gamma$), the rational lemniscate $\Gamma$ is a shape and we can consider its fingerprint $k\colon \bT\to\bT$. If $\Phi_\pm$ denote the corresponding Riemann maps $\bD_\pm\to \Omega_\pm$ as above (where we assume, say, that $\infty\in \Omega_+$, since otherwise the strategy of defining the fingerprint would need to be slightly modified), then the fingerprint would satisfy the functional equation $\Phi_+\circ k=\Phi_-$. Composing with $R$ to the left on both sides, we obtain an equation of the form $A\circ k=B$, where $A=R\circ \Phi_+$ and $B=R\circ \Phi_-$ are Blaschke products of degree $n$. In the case of polynomial lemniscates considered in this paper,  we have $A(z)=z^n$. We suspect that all diffeomorphisms $k\colon \bT\to \bT$ that arise from the algebraic equation $A\circ k=B$, where $A$ and $B$ are Blaschke products of degree $n$, are fingerprints of rational lemniscates. A first obstacle in this study would be to establish an analytic criterion for when the rational lemniscate $\Gamma=\{z\in \hat \bC\colon |R(z)|=1\}$ has a connected and simply connected interior $\Omega_-$. In the polynomial case, the criterion for $\Omega_-$ to be connected and simply connected is that the $n-1$ finite critical values of the polynomial $P$ lie in the unit disk (Proposition \ref{propcrit}). In the case of a more general rational lemniscate, it is easily seen that this is also a necessary but not a sufficient condition. However, one can show that if $R$ has $n-1$ critical values in $\bD$ {\it and} $\Omega_-$ is connected {\it or} every component of $\Omega_-$ is simply connected, then $\Omega_-$ is both connected and simply connected. The authors hope to return to the case of rational lemniscates in a future paper.

\bibliographystyle{amsplain}
\bibliography{td2}

\end{document}